\documentclass[reqno]{amsart}
\usepackage{amsmath,amssymb}
\usepackage{hyperref}
\newcommand{\arxiv}[1]{\href{http://arxiv.org/abs/#1}{arXiv:#1}}
\newcommand*{\mailto}[1]{\href{mailto:#1}{\nolinkurl{#1}}}

\newcommand{\nn}{\nonumber}
\newcommand{\be}{\begin{equation}}
\newcommand{\ee}{\end{equation}}
\newcommand{\ol}{\overline}
\newcommand{\spr}[2]{\langle #1 , #2 \rangle}
\DeclareMathOperator{\ran}{Ran}

\DeclareMathOperator{\tr}{tr}

\DeclareMathOperator*{\slim}{s-lim}

\newcommand{\eps}{\varepsilon}

\newcommand{\gam}{\gamma}
\newcommand{\sig}{\sigma}
\newcommand{\lam}{\lambda}

\newcommand{\id}{{\mathbb I}}
\newcommand{\I}{\mathrm{i}}
\newcommand{\E}{\mathrm{e}}
\newcommand{\R}{{\mathbb R}}
\newcommand{\N}{{\mathbb N}}

\newcommand{\C}{{\mathbb C}}

\newcommand{\clo}{\mathfrak{C}}
\newcommand{\hr}{\mathfrak{H}}
\newcommand{\db}{\mathfrak{D}}
\newcommand{\qb}{\mathfrak{Q}}

\newtheorem{theorem}{Theorem}[section]

\newtheorem{lemma}[theorem]{Lemma}
\newtheorem{corollary}[theorem]{Corollary}
\theoremstyle{definition}
\newtheorem{example}[theorem]{Example}
\newtheorem{remark}[theorem]{Remark}
\newtheorem{hypothesis}[theorem]{Hypothesis}

\numberwithin{equation}{section}

\begin{document}

\title[Generalized Resolvent Convergence]{On Generalized Strong and Norm Resolvent Convergence}

\author[G. Teschl]{Gerald Teschl}
\address{Faculty of Mathematics\\ University of Vienna\\
Oskar-Morgenstern-Platz 1\\ 1090 Wien\\ Austria}
\email{\mailto{Gerald.Teschl@univie.ac.at}}
\urladdr{\url{https://www.mat.univie.ac.at/~gerald/}}

\author{Yifei Wang}
\address{State Key Laboratory of Mathematical Sciences\\ Academy of Mathematics and Systems Science\\  Chinese Academy of Sciences\\ 100190  Beijing\\ China\\ School of Mathematical Sciences\\ University of Chinese Academy of Sciences\\ 100049 Beijing\\ China}
\email{\mailto{wangyifei@amss.ac.cn}}

\author{Bing Xie}
\address{School of Mathematics and Statistics\\ Shandong University\\ 264209 Weihai\\ China}
\email{\mailto{xiebing@sdu.edu.cn}}

\author{Zhe Zhou}
\address{State Key Laboratory of Mathematical Sciences\\ Academy of Mathematics and Systems Science\\  Chinese Academy of Sciences\\ 100190  Beijing \\ China}
\email{\mailto{zzhou@amss.ac.cn}}

\begin{abstract} 
We present a streamlined approach for generalized strong and norm convergence of self-adjoint operators
in different Hilbert spaces. In particular, we establish convergence of associated (semi-)groups, (essential) spectra
and spectral projections. In addition, we give some applications to Sturm--Liouville operators.
\end{abstract}

\maketitle

\section{Introduction}

The standard notions of convergence for unbounded self-adjoint operators are strong and norm resolvent convergences. In the classical setting, all operators are defined on the same Hilbert space. However,
there are many situations of practical interest where two operators are defined on different
Hilbert spaces. For example, if one studies the case where an operator is restricted to
a sequence of subspaces. To accommodate such situations, it is natural to work with a corresponding sequence
of embedding operators and this is commonly known as generalized strong/norm resolvent convergences
and dates at least back to Stolz and Weidmann \cite{SW93,SW95} who applied this to the
approximation of singular Sturm--Liouville operators with regular ones. Further extensions were given
by Weidmann in his monograph \cite[Section~9.3]{weidmann1} (see also \cite{T08} and \cite[Section~6.6]{tschroe}). Other applications are
the case of Sturm--Liouville operators (or, more general, canonical systems, elliptic partial differential operators) with
varying weight functions or the approximation of Schr\"odinger operators on different graphs.
In particular, for the last purpose, generalized strong/norm resolvent convergences was further
extended by Post in \cite{P06} and in his monograph \cite[Chapter~4]{post}. See also the recent overview by Post and Simmer \cite{PS25}.
Furhtermore, perturbation theory for such situations also plays an important role in the study of indefinite Sturm--Liouville operators \cite{BSTT24}.
Our main motivation here stems from applications to Sturm--Liouville operators with varying weight functions.

The main purpose of the present paper is to advocate a streamlined approach to both generalized
strong and norm resolvent convergence. Our motivation is to reduce the number of requirements for the embedding operators and
only work with one natural condition, equations \eqref{Jcos} and \eqref{Jcon}, for generalized strong and norm resovent convergence,
respectively. In contradistinction to Stolz and Weidmann, we do not require our embedding operators to be projections.
While at least in the case of strong resolvent convergence this seems frequent, the present slightly more general approach seems natural
also in this context and does not impose any additional difficulties. In contradistinction to Post, we work with a reduced number of requirements and
we do not introduce a quantified concept of $\delta$-quasi-unitary, since we are only interested in limiting results. This
allows for a streamlined approach providing the (from our point of view) main results, including full proofs, in just seven pages.
While overall we follow the presentation of \cite[Section~6.6]{tschroe} (which in turn is inspired by \cite[Section~9.3]{weidmann1}),
we have tried to further simplify proofs (to compensate for the extra complications added by the embedding operators) and we
also include some amplifications, which even go beyond classical results at a few points. In addition, again in contradistionction
to Post, we do not require semiboundedness for our operators.

In our final section we will provide some applications of these results and show how to improve some recent results which have been
obtained using different methods.

Overall we hope that this result will foster further research in this direction as we consider this as an important topic in
operator theory which deserves more attention. 

\section{Generalized strong and norm resolvent convergence}

Let $\hr$ and $\hr_n$, $n\in\N$, be separable Hilbert spaces together with bounded operators $J_n\in\mathcal{B}(\hr, \hr_n)$ such that
\be\label{Jcon}
\|J_n \| \le C, \qquad \|J_n^* J_n - \id \| = \|J_n J_n^* - \id \|\to 0,
\ee
where $C$ is a positive constant and $J_n^*$ is the adjoint operator of $J_n$.
While this is similar to \cite[Chapter~4]{post}, we have opted for a slightly more streamlined definition, which is simpler
and more natural for our intended purpose.
In this connection note that \eqref{Jcon} implies that $J_n^* J_n$ and $J_n J_n^*$ are both invertible for sufficiently large $n$. Hence, it is not hard to see that both $J_n^{-1} = J_n^*(J_n J_n^*)^{-1}$ and $(J_n^*)^{-1} = (J_n J_n^*)^{-1} J_n$ exist in this case with $\|J_n^* - J_n^{-1}\|\to 0$,
$\|J_n\|=\|J_n^*\|\to 1$ and $\|J_n^{-1}\|=\|(J_n^*)^{-1}\|\to 1$.

We will also look at the case where
\be\label{Jcos}
\|J_n \| \le C, \qquad \|J_n^* J_n \psi - \psi \| \to 0, \quad \psi\in\hr.
\ee
In this case, $J_n$ does not have to be eventually invertible and we only have that $\liminf\limits_{n\rightarrow \infty} \|J_n\|\ge 1$. However, in the applications
we have in mind, $J_n$ will be surjective and hence $J_n^*$ will have a bounded inverse by Banach's closed range theorem (the
converse being also true).

\begin{example}\label{jn}
The first example we have in mind are weighted Lebesgue spaces $\hr :=  L^2(\R^d, w_\infty(x) \mathrm{d}x)$, $\hr_n :=  L^2(\R^d, w_n(x) \mathrm{d}x)$
with $w_n(x) \le C w_\infty(x)$ and $\|\frac{w_n}{w_\infty} - 1\|_\infty \to 0$, where we set $\frac{0}{0} : =1$.
In this case $J_n \psi = \psi$ is simply the natural embedding and $J_n^*$ is multiplication with $\frac{w_n}{w_\infty}$.
Equation \eqref{Jcon} will hold if $\frac{w_n}{w_\infty} \to 1$ uniformly.

The other example is where $\hr_n \subseteq \hr$ is a sequence of subspaces. In this case $J_n$ is the corresponding orthogonal projection
and $J_n^*$ is the natural embedding. Note that we have $J_n J_n^* = \id$ and \eqref{Jcos} will hold if $J_n\to \id$ strongly.
\end{example}

Let $A\in\clo(\hr)$, $A_n\in\clo(\hr_n)$ be (densely defined) closed operators. Then we say that $A_n$ converges to $A$ in generalized norm resolvent sense if \eqref{Jcon} holds and
\be\label{defnrc}
\lim\limits_{n\to\infty} \|J_n^* R_{A_n}(z) J_n - R_A(z) \| = 0
\ee
for one $z\in \Gamma=\C\setminus\Sigma$, $\Sigma:=\ol{\sig(A)\cup \bigcup_n \sig(A_n)}$. Using \eqref{Jcon} this is equivalent to
\be\label{defnrc2}
\lim\limits_{n\to\infty} \|R_{A_n}(z) J_n - J_n R_A(z) \| = 0.
\ee
Similarly, we say that $A_n$ converges to $A$ in generalized strong resolvent sense if \eqref{Jcos} holds and
\be\label{defsrc}
\lim\limits_{n\to\infty} \|J_n^* R_{A_n}(z) J_n \psi - R_A(z) \psi \| = 0, \qquad \psi\in\hr,
\ee
for one $z\in \Gamma$. As before, by \eqref{Jcos}, this is implied by
\be\label{defsrc2}
\lim\limits_{n\to\infty} \|R_{A_n}(z) J_n \psi - J_n R_A(z) \psi \| = 0, \qquad \psi\in\hr,
\ee
but the converse might not be true unless $J_n$ is surjective.

\begin{remark}
In the classical case, one often avoids taking the closure in the definition of $\Sigma$. Indeed, since strong/norm convergent sequences are bounded, a $z$ for which \eqref{defnrc} or \eqref{defsrc} hold will always be in the interior of $\Gamma$ in this case. In our situation, this is still
true for \eqref{defnrc} since $J_n$ is invertible as pointed out before. It is also true for \eqref{defsrc} if $J_n$ is an orthogonal
projection or if $J_n$ is invertible. However, in general it could fail and it does not seem worth to add extra conditions to include these
additional points in $\Gamma$.
\end{remark}

\begin{theorem}\label{fnorm}
Let $A_n$, $A$ be self-adjoint operators and suppose that $A_n$ converges to $A$ in  generalized norm resolvent sense.
Then 
\be
\lim_{n\to\infty} J_n^* f(A_n) J_n = \lim_{n\to\infty} J_n^{-1} f(A_n) J_n = f(A)
\ee
for every bounded continuous function $f\!:\! \Sigma\! \to \! \C$ with
$\lim\limits_{\lam\to-\infty}f(\lam)=\lim\limits_{\lam\to\infty}f(\lam)$.

If $A_n$ converges to $A$ in  generalized strong resolvent sense, then
\be
\slim_{n\to\infty} J_n^* f(A_n)J_n = f(A)
\ee
for every bounded continuous function $f: \Sigma \to \C$.
\end{theorem}

\begin{proof}
The set of functions for which the claim holds forms a $*$-subalgebra.
That it is a subalgebra follows since addition, scalar multiplication and products are continuous with respect to norm/strong
convergence. To see that the claim holds for $f^*$ if it holds for $f$, use that
$\|J_n^* f(A_n) J_n - f(A) \| = \|J_n^* f^*(A_n) J_n - f^*(A) \|$ in the case of norm convergences. In the case of strong
convergence, note that $J_n^* f(A_n) J_n \psi \to f(A) \psi$ implies weak convergence $J_n^* f^*(A_n) J_n \psi \rightharpoonup f^*(A) \psi$
and $\|J_n^* f(A_n) J_n \psi\| = \|J_n^* f^*(A_n) J_n \psi\| \to \|f(A) \psi\| = \|f^*(A) \psi\|$ since $f(A_n)$ and $f(A)$ are normal.
Together, both show $J_n^* f^*(A_n) J_n \psi \to f^*(A) \psi$.

Since this $*$-subalgebra contains $f(\lam)=1$ and $f(\lam)=\frac{1}{\lam-z_0}$, $z_0\in\Gamma$, it is dense by the Stone--Weierstra{\ss} theorem.
The usual $\frac{\eps}{3}$ argument shows that this $*$-subalgebra is also closed. 
\end{proof}

Two easy consequences are worth noticing:

\begin{corollary}
Suppose $A_n$ converges to $A$ in  generalized strong or norm resolvent sense for
one $z_0\in\Gamma$. Then this holds for all $z\in\Gamma$.
\end{corollary}

\begin{corollary}
Suppose $A_n$ converges to $A$ in  generalized strong resolvent sense. Then
\be
J_n^* \E^{\I t A_n} J_n \overset{s}{\to} \E^{\I t A}, \qquad t\in\R,
\ee
and if all operators are semi-bounded by the same bound
\be
J_n^* \E^{-t A_n} J_n \overset{s}{\to} \E^{-t A}, \qquad t\ge 0.
\ee
The last claim also holds in norm, provided we have generalized norm resolvent convergence.
\end{corollary}

Next we turn to criteria to verify resolvent convergence. The first one deals with the case
when the difference is relatively bounded.

\begin{lemma}
Let $A_n$, $A$ be self-adjoint operators with $J_n \db(A) \subseteq \db(A_n)$. Then
$A_n$ converges to $A$ in  generalized  norm resolvent sense if there
are non-negative sequences $a_n$ and $b_n$ converging to zero such that
\be
\|(A_n J_n -J_n A)\psi\| \le a_n \|A \psi\| + b_n \|\psi\|, \quad \psi\in\db(A).
\ee
\end{lemma}

\begin{proof}
From the second resolvent formula
\be\label{eqsndrf}
R_{A_n}(z) J_n - J_n R_A(z) = R_{A_n}(z) (J_n A-A_n J_n) R_A(z), \quad z\in\Gamma,
\ee
we infer
\begin{align*}
\|(R_{A_n}(\I) J_n - J_n R_A(\I))\psi \| &\le \|R_{A_n}(\I)\| \Big(
a_n\| A R_A(\I) \psi\| + b_n \| R_A(\I) \psi\|\Big)\\
&\le ( a_n + b_n) \|\psi\|
\end{align*}
and hence $\|R_{A_n}(\I) J_n - J_n R_A(\I)\| \le a_n + b_n \to 0$.
\end{proof}

The second deals with the case when the difference is relatively form bounded.

Let $A\in\clo(\hr)$ be a closed linear operator. Then the expression $\spr{\psi}{A\psi}$ is called the quadratic form,
$$
\textbf{q}_A(\psi)= \spr{\psi}{A\psi},\quad \psi\in\db(A),
$$
associated to $A$. The operator $A$ is called nonnegative if for all $\psi\in\db(A)$ we have $\textbf{q}_A(\psi)\geq 0$. For a nonnegative operator $A$, introduce the scalar product
$$
\spr{\varphi}{\psi}_A:= \spr{\varphi}{(A+1)\psi},\quad \forall\;\varphi,\psi\in\db(A).
$$
Let $\hr_A$ be the completion of $\db(A)$ with respect to the above scalar product. It is easy to verify that $\db(A)\subseteq\hr_A\subseteq\hr$.
Clearly, the quadratic form $\textbf{q}_A$ can be extended to every $\psi\in\hr_A$ by setting
$$
\textbf{q}_A(\psi)= \spr{\psi}{\psi}_A-\|\psi\|^2,\quad \psi\in\qb(A)=\hr_A.
$$
The set $\qb(A)$ is also called the form domain of $A$. For more details, see \cite{tschroe}.

\begin{lemma}\label{q-nrc}
Suppose $A, A_n \ge \gam$ are bounded from below with $J_n \qb(A) \subseteq \qb(A_n)$. Then
$A_n$ converges to $A$ in   generalized  norm resolvent sense if there
are non-negative sequences $a_n$ and $b_n$ converging to zero such that
\[
|\textbf{q}_A(\psi) - \textbf{q}_{A_n}(J_n \psi)| \le a_n \textbf{q}_{A-\gam}(\psi) + b_n \|\psi\|^2, \qquad \psi\in\qb(A),
\]
where $\textbf{q}_A$ and $\textbf{q}_{A_n}$ are the associated quadratic forms.
\end{lemma}

\begin{proof}
We can assume $a_n, b_n <1$ without loss of generality. Applying \cite[Theorem 6.31]{tschroe} with $\textbf{q}(\psi)=\textbf{q}_{A_n}(J_n\psi)-\textbf{q}_A(\psi)$ for any $\psi\in\qb(A)$ we have 
$$\|C_{\textbf{q}}(\lambda)\|\leq \max\left( a_n, \frac{b_n}{\lambda+\gam}\right)\rightarrow 0,\quad \lam>-\gam,$$
and
$$R_{B_n}(-\lam)=R_A(-\lam)^{\frac{1}{2}}(\id + C_{\textbf{q}}(\lam))^{-1}R_A(-\lam)^{\frac{1}{2}},$$
where $C_{\textbf{q}}(\lam)$  is the  bounded operator corresponding to the quadratic form $\textbf{q}(R_A(-\lam)^{\frac{1}{2}}\cdot)$ and $B_n:=J_n^*A_n J_n$ is the self-adjoint operator associated with the quadratic form $\textbf{q}_{A_n}(J_n\cdot)$.
Hence we have
\begin{align*}
\|R_{B_n}(-\lam)-R_{A}(-\lam)\| &= \|R_A(-\lam)^{\frac{1}{2}}\left((1+C_{\textbf{q}}(\lam))^{-1}-1 \right) R_A(-\lam)^{\frac{1}{2}}\|\\
&\leq \|R_A(-\lam)^{\frac{1}{2}}\|^2\|(1+C_{\textbf{q}}(\lam))^{-1}-1\|\rightarrow 0.
\end{align*}
Moreover, we have
\begin{align*}
& \| J_n R_{B_n}(-\lam) - R_{A_n}(-\lam) J_n \| = \| R_{A_n}(-\lam)(A_n+\lam) (J_n R_{B_n}(-\lam) - R_{A_n}(-\lam) J_n) \|\\
& \qquad \le \| R_{A_n}(-\lam)\| \| (A_n + \lam) J_n R_{B_n}(-\lam) - J_n \|\\
& \qquad \le \| R_{A_n}(-\lam)\| \|(J_n^*)^{-1} \| \| (B_n +\lam +\lam(J_n^* J_n - \id)) R_{B_n}(-\lam) - J_n^* J_n\|\\
& \qquad \le \| R_{A_n}(-\lam)\| \|(J_n^*)^{-1} \| \| (\lam R_{B_n}(-\lam) - 1)\| \|J_n^* J_n - \id\|
\to 0, \qedhere
\end{align*}
from which the claim follows.
\end{proof}

Next, we establish that strong convergence of the operators implies strong resolvent convergence.

Let $A$ be a symmetric operator with dense domain $\db(A)$ and $\overline{A}$ is its only self-adjoint extension. Then $A$ is called essentially self-adjoint and $\db(A)$ is called a core for $\overline{A}$. For more details, see \cite{tschroe}.

\begin{lemma}
The sequence $A_n$ converges to $A$ in generalized strong resolvent sense if there
is a core $\db_0$ of $A$ such that for every $\psi\in\db_0$
we have $J_n \psi\in\db(A_n)$ for $n$ sufficiently large and
$(A_n J_n - J_n A) \psi \to 0$.
\end{lemma}

\begin{proof}
Using the second resolvent formula \eqref{eqsndrf}, we have
\[
(R_{A_n}(\I) J_n - J_n R_A(\I)) \phi = -R_{A_n}(\I) (A_n J_n - J_n A) \psi\to 0
\]
for $\phi = (A-\I) \psi$ with $\psi\in\db_0$ which is dense, since $\db_0$ is a core.
The rest follows from \cite[Lemma~1.14]{tschroe}.
\end{proof}

Finally, we mention that generalized weak resolvent convergence, is equivalent to the generalized strong resolvent convergence.

\begin{lemma}
Let $A_n$, $A$ be self-adjoint operators.
Suppose \eqref{Jcos} and $\|J_n J_n^* - \id \|\to 0$. If
$J_n^* R_{A_n}(z) J_n$ converges to $R_A(z)$ weakly for some $z\in\Gamma\setminus\R$,
then $R_{A_n}(z)$ converges to $R_A(z)$ in generalized strong resolvent sense.
\end{lemma}

\begin{proof}
By $J_n^*R_{A_n}(z)J_n \rightharpoonup R_A(z)$ we also have
$J_n^*R_{A_n}(z)^*J_n \rightharpoonup R_A(z)^*$ and thus by the first resolvent formula
\begin{align*}
& \left|\|J_n^*R_{A_n}(z)J_n \psi\|^2 - \|R_A(z) \psi\|^2\right| \\
&=\left|\spr{\psi}{J_n^*R_{A_n}(z^*)J_nJ_n^*R_{A_n}(z)J_n\psi - R_A(z^*) R_A(z) \psi}\right|\\
&\leq |\spr{\psi}{J_n^*R_{A_n}(z^*)(J_nJ_n^*-\id)R_{A_n}(z)J_n\psi}|\\
&\qquad +|\spr{\psi}{(J_n^*R_{A_n}(z^*)R_{A_n}(z)J_n- R_A(z^*) R_A(z))\psi}|\\
&\leq \|J_n^*\| \|R_{A_n}(z^*)\| |\spr{\psi}{(J_nJ_n^*-\id)R_{A_n}(z)J_n\psi}| \\
&\qquad + \frac{1}{|z-z^*|} \left|\spr{\psi}{(J_n^*R_{A_n}(z)J_n)\psi - (J_n^*R_{A_n}(z)^*J_n) \psi} \right|\\
&\to 0.
\end{align*}
Together with $J_n^*R_{A_n}(z)J_n \rightharpoonup R_A(z)$ we have
$J_n^*R_{A_n}(z)J_n\psi \to R_A(z) \psi$ by virtue of \cite[Lemma~1.12 (iv)]{tschroe}.
\end{proof}

Now we turn to convergence of spectra. As usual, we denote by $\lim\limits_{n\to\infty} S_n$ the set of all $\lam$
for which there is a sequence $\lam_n\in S_n$ converging to $\lam$.

\begin{theorem}\label{spe-cov}
Let $A_n$ and $A$ be self-adjoint operators. If $A_n$ converges to $A$
in  generalized strong resolvent sense, we have $\sig(A) \subseteq \lim_{n\to\infty} \sig(A_n)$.

If $A_n$ converges to $A$ in  generalized  norm resolvent sense, we have $\sig(A) =
\lim\limits_{n\to\infty} \sig(A_n)$.
Moreover, for every $\lam\in\rho(A)$ there is a
neighborhood which is eventually contained in $\rho(A_n)$.
\end{theorem}

\begin{proof}
First, we claim $\sig(A)\subseteq \lim\limits_{n\to\infty} \sig(A_n)$. If not, then we could find a $\lam\in \sig(A)$ and some $\eps>0$ such that 
\[
\sig(A_n)\cap (\lam-2\eps,\lam+2\eps)=\emptyset.
\]
Choose a nonnegative bounded continuous function $f$ which takes the value $1$ on $(\lam-\eps,\lam+\eps)$ and vanishes outside $(\lam-2\eps,\lam+2\eps)$. Then $f(A_n)=0$. Combining with Theorem~\ref{fnorm}, we have
\[
J_n^* f(A_n)J_n\psi-f(A)\psi\rightarrow 0\quad\text{for any $\psi\in \db(A)$},
\]
and thus $f(A)\psi=0$. On the other hand, since $\lam\in \sig(A)$, there is a non-zero $\psi\in\mathrm{Ran}P_A((\lam-\eps,\lam+\eps))$ implying $f(A)\psi=\psi$, which is a contradiction.

To get the converse direction we choose $\lam\in\rho(A)$ and $\eps>0$ such that $\sig(A)\cap (\lam-2\eps,\lam+2\eps)=\emptyset$.
Choose a nonnegative bounded continuous function $f$ which takes the value $1$ on $(\lam-\eps,\lam+\eps)$ and vanishes outside $(\lam-2\eps,\lam+2\eps)$. Then $f(A)=0$ and by $J_n^* f(A_n)J_n\to 0$ we also have $f(A_n) \to 0$ and
\[
f(A_n) P_{A_n}((\lam-\eps,\lam+\eps)) = P_{A_n}((\lam-\eps,\lam+\eps)) \to 0.
\]
This implies $P_{A_n}((\lam-\eps,\lam+\eps))=0$ eventually and establishes the claim.
\end{proof}

In addition we have the following results for the convergence of spectral projections.

\begin{lemma}
Suppose $A_n$ converges in  generalized strong resolvent sense to $A$. If
$P_A(\{\lam\})=0$, then
\be
\slim_{n\to\infty} J_n^* P_{A_n}((-\infty,\lam)) J_n =
\slim_{n\to\infty} J_n^* P_{A_n}((-\infty,\lam]) J_n =
P_A((-\infty,\lam)) = P_A((-\infty,\lam]).
\ee
\end{lemma}

\begin{proof}
By Theorem~\ref{fnorm}, the spectral measures $\mu_{n,J_n\psi}$ corresponding to $A_n$
converge vaguely to the spectral measure $\mu_\psi$ of $A$. Hence $\|P_{A_n}(\Omega) J_n \psi\|^2 = \mu_{n,J_n \psi}(\Omega)$
together with \cite[Lem.~A.34]{tschroe} implies the claim.
\end{proof}

Using $P((\lam_0,\lam_1))= P((-\infty,\lam_1)) - P((-\infty,\lam_0])$, we also obtain

\begin{lemma}\label{lemliminfproj}
Suppose $A_n$ converges in the generalized strong resolvent sense to $A$. If
$P_A(\{\lam_0\})=P_A(\{\lam_1\})=0$, then
\be
\slim_{n\to\infty} J_n^* P_{A_n}((\lam_0,\lam_1)) J_n =
\slim_{n\to\infty} J_n^* P_{A_n}([\lam_0,\lam_1]) J_n =
P_A((\lam_0,\lam_1)) = P_A([\lam_0,\lam_1]).
\ee
Moreover,
\begin{align}\nn
\dim\ran P_A((\lam_0,\lam_1)) &\le \liminf_{n\to\infty} \dim\ran J_n^* P_{A_n}((\lam_0,\lam_1)) J_n\\
&\le \limsup_{n\to\infty} \|J_n\|^2 \liminf_{n\to\infty} \dim\ran P_{A_n}((\lam_0,\lam_1)).
\end{align}
\end{lemma}

\begin{proof}
It remains to show the inequalities. Since we can write
\[
\dim\ran P_{A_{n_k}}((\lam_0,\lam_1)) = \tr(P_{A_{n_k}}((\lam_0,\lam_1)))
\]
the first follows from strong convergence of the spectral projectors together with Fatou's lemma.
The second inequality follows since for every positive operator $P$ of rank $N$ we have $\tr(P) \le N \|P\|$.
\end{proof}

\begin{corollary}
Suppose $A_n$ converges in the generalized strong resolvent sense to $A$. If $\lam\in\sig_{ess}(A)$, then for
every $\eps>0$ we have
\be
\lim_{n\to\infty} \dim\ran P_{A_n}((\lam-\eps,\lam+\eps)) = \infty.
\ee
\end{corollary}

\begin{proof}
By decreasing $\eps$ if necessary we can assume $P_A(\{\lam-\eps\})=P_A(\{\lam+\eps\})=0$. Hence the claim claim follows from
the previous lemma since the assumption implies $\dim\ran P_A((\lam-\eps,\lam+\eps)) = \infty$ for every $\eps>0$.
\end{proof}

For our next result recall the following well-known results:
If $P$, $Q$ are two projections with $\|P-Q\|<1$, then the dimension of their ranges agree.

\begin{lemma}\label{project}
Suppose $A_n$ converges to $A$ in generalized norm resolvent sense and $\lam$ is an isolated point of $\sig(A)$. Then
\be
\lim_{n\to\infty} J_n^{-1} P_{A_n}((\lam-\eps,\lam+\eps)) J_n = P_A(\{\lam\}),
\ee
for $\eps$ so small that $(\lam-\eps,\lam+\eps) \cap \sig(A) = \{\lam\}$. In particular, the dimensions of the ranges of both projections eventually agree.
\end{lemma}

\begin{proof}
Let $f$ be a continuous function which is $1$ one $(\lam-\eps/2,\lam+\eps/2)$ and zero outside $(\lam-\eps,\lam+\eps)$.
Then, for sufficiently large $n$ we have that $\sig(A_n) \cap \{ \alpha \mid \eps/2 < |\alpha -\lam| < \eps \} = \emptyset$.
Hence $f(A_n) = P_{A_n}((\lam-\eps,\lam+\eps))$ and $f(A) = P_A(\{\lam\})$ such that $J_n^{-1} P_{A_n}((\lam-\eps,\lam+\eps)) J_n \to P_A(\{\lam\})$ and the claim follows.
\end{proof}

\begin{theorem}
Let $A_n$ and $A$ be self-adjoint operators. If $A_n$ converges to $A$ in  generalized  norm resolvent sense, we have $\sig_{ess}(A) =
\lim\limits_{n\to\infty} \sig_{ess}(A_n)$.
Here $\lim\limits_{n\to\infty} \sig_{ess}(A_n)$ denotes the set of all $\lam$ for which there is a
sequence $\lam_n\in\sig_{ess}(A_n)$ converging to $\lam$.
\end{theorem}
\begin{proof}
If $\lam\notin\sig_{ess}(A)$, then $\lam$ is in the discrete spectrum (or in the resolvent set) and the claim follows from
Lemma~\ref{project}.

Conversely, let $\lam\in\sig_{ess}(A)$. If there were no sequence $\lam_n \in\sig_{ess}(A_n)$ converging to $\lam$,
we could find $\eps>0$ and a subsequence $n_k$ such that $\dim \ran P_{A_{n_k}}((\lam-\eps,\lam+\eps))<\infty$.
We pass to the subsequence for notational convenience.
Moreover, we can choose a singular Weyl sequence $\psi_k$ with $\|\psi_k\| =1$, $\psi_k\rightharpoonup 0$, and
$\|R_A(\I) - (\lam-\I)^{-1})\psi_k\|\to 0$. Then
\begin{align*}
&\| J_n^* R_{A_n}(\I) J_n - R_A(\I)\| \ge \limsup_{k\to\infty} \| (J_n^* R_{A_n}(\I) J_n - R_A(\I)) \psi_k\|\\
& \qquad = \limsup_{k\to\infty} \| (J_n^* R_{A_n}(\I) J_n - (\lam-\I)^{-1}) \psi_k\|\\
& \qquad = \limsup_{k\to\infty} \| J_n^* (R_{A_n}(\I) - (\lam-\I)^{-1}) J_n \psi_k\|\\
& \qquad = \limsup_{k\to\infty} \| J_n^* (R_{A_n}(\I) - (\lam-\I)^{-1}) P_{A_n}(\R\setminus(\lam-\eps,\lam+\eps) )J_n \psi_k\|\\
& \qquad \ge \|(J_n^*)^{-1} \| \frac{\eps}{\sqrt{1+\eps^2}} \|J_n^{-1} \|
\end{align*}
contradicting norm resolvent convergence. Here we have used that $J_n^*P_{A_n}((\lam-\eps,\lam+\eps) J_n \psi_k \to 0$ by compactness.
\end{proof}

\section{Applications to Sturm--Liouville operators}

In this section, we apply the generalized norm resolvent convergence to  Sturm--Liouville operators in different Hilbert spaces. 
Throughout this section we will make the usual assumptions on the coefficients:
\begin{hypothesis}\label{hyp}
Assume that $I=(a,b)$ is some interval. Let $w,p^{-1}, q \in  L^1_{loc}( I,\mathrm{d}x)$ be real-valued and $p, w>0$ a.e. in $I$.
\end{hypothesis}

Then we consider the Sturm--Liouville operator:
\begin{equation}
    \begin{split}
        T:\quad \db(T) &\rightarrow  L^2( I,w(x)\mathrm{d}x)\\
        f&\mapsto \frac{1}{w}\bigg(-\frac{\mathrm{d}}{\mathrm{d}x}p\frac{\mathrm{d}f}{\mathrm{d}x}+qf\bigg),
    \end{split}
\end{equation}
where
\be
\db(T) \subseteq \{f\in L^2(I,w(x)\mathrm{d}x):f,pf^{\prime}\in AC( I),\frac{1}{w}\big(-(pf^\prime)^\prime+qf\big)\in  L^2(I,w(x)\mathrm{d}x)\}.
\ee
In order for $T$ to be self-adjoint, one might have to impose additional boundary conditions, but we postpone this issue for
now.

For our present purpose, we will need the corresponding form domain. Most results we are aware of consider the case of Schr\"odinger
operators ($w=p=1$), see for example \cite{tschroe}, or the case of regular operators, see \cite{GST96}. To compute
the form domain in a more general situation, we will first set $q$ equal to $0$ and factorize $T$ according to $A^*A$
such that we get $\qb(T)=\db(A)$.

To this end define \begin{equation*}
\begin{split}
    A_+:\quad \db(A_+) &\rightarrow L^2( I,w(x)\mathrm{d}x)\\
    f&\mapsto\frac{1}{w} \sqrt{w p} \frac{\mathrm{d}f}{\mathrm{d}x},
\end{split}
\end{equation*}
where $\db(A_+):= \{ f \in  L^2( I,w(x)\mathrm{d}x): f \in  AC( I), \sqrt{p} f^\prime \in  L^2( I,\mathrm{d}x) \}$
and
\begin{equation*}
\begin{split}
    A_-:\quad \db(A_-) &\rightarrow L^2( I,w(x)\mathrm{d}x)\\
    f&\mapsto-\frac{1}{w} \frac{\mathrm{d}}{\mathrm{d}x} \sqrt{w p}f,
\end{split}
\end{equation*}
where $\db(A_-):= \{ f \in  L^2( I,w(x)\mathrm{d}x):\sqrt{w p} f \in  AC( I), \frac{1}{\sqrt{w}}(\sqrt{w p} f )^\prime \in L^2( I,\mathrm{d}x) \}$.

Moreover, we will also need the associated minimal operators $A_{0,\pm}$ which are given by $A_\pm$ restricted to
functions with compact support. 

Note that for $(c,d)\subseteq (a,b)$, $f\in\db(A_+)$ and $g\in\db(A_-)$ integration by parts shows
\begin{align}\nn
& \int_c^d g(x) \sqrt{w(x)p(x)} f'(x) \mathrm{d}x\\ \label{eqintbpart}
& \quad = (\sqrt{w p} g)(d) f(d) - (\sqrt{w p} g)(c) f(c) - \int_c^d (\sqrt{w p} g)'(x) f(x) \mathrm{d}x.
\end{align}
Moreover, since all integrands are integrable, the limits of the boundary terms exist as $c\downarrow a$ and $d \uparrow b$.

Also, if $g\in  L^2( I, w(x)\mathrm{d}x)$ then we have that
\be\label{eqfp}
f_+(x) := f_+(c) + \int_c^x \sqrt{\frac{w(y)}{p(y)}} g(y) \mathrm{d}y
\ee
satisfies $f_+\in  AC( I)$ and $\sqrt{p/w}(f_+)' = g$ and
\be\label{eqfm}
f_-(x) = \frac{1}{\sqrt{p(x) w(x)}} \left( (\sqrt{w p} f_-)(c) + \int_x^c w(y) g(y) \mathrm{d}y \right)
\ee
satisfies $\sqrt{w p} f_-\in  AC( I)$ and $-1/w(\sqrt{w p} f_-)' = g$.

The following type of result is well-known. We include a proof for the sake of completeness.

\begin{lemma}
    The operators $A_{0,\pm}$ are densely defined and their closures are given by 
    $$ \ol{A_{0,+}}f=A_{+}f,\quad \db(\ol{A_{0,+}}) = \{ f \in \db(A_+): \lim_{x\to a,b} (\sqrt{w p} g)(x) f(x) = 0,\;\forall g\in\db(A_-)\}$$
    and
    $$ \ol{A_{0,-}}f=A_{-}f,\quad \db(\ol{A_{0,-}}) = \{ f \in \db(A_-): \lim_{x\to,a,b} (\sqrt{w p} f)(x) g(x) = 0,\;\forall g\in\db(A_+)\}.$$
    Their adjoint operators are given by
    $$A_{0,\pm}^*f=A_{\mp}f,\quad \db(A_{0,\pm}^*)=\db(A_{\mp}).$$
\end{lemma}
\begin{proof}
By \eqref{eqintbpart} we have $\db(A_\mp)\subseteq \db(A_{0,\pm}^*)$
and it remains to show $\db(A_{0,\pm}^*)\subseteq \db(A_\mp)$. If we have
\[
\spr{h}{A_{0,\pm} f} = \spr{k}{f}, \qquad \forall f\in \db(A_{0,\pm}),
\]
then by \eqref{eqfp}, \eqref{eqfm} we can find a $\tilde{h}$ which is locally in $\db(A_\mp)$ such that $A_\mp \tilde{h} = k$ and integration by parts shows
\[
\int_a^b (h(x)-\tilde{h}(x)^*) A_\pm f(x) \mathrm{d}x =0, \qquad \forall f\in \db(A_{0,\pm}).
\]
Now consider the linear functionals
\[
\ell(g)= \int_a^b (h(x)-\tilde{h}(x))^* g(x) w(x) \mathrm{d}x, \quad
\ell_\pm(g)= \int_a^b u_\mp(x)^* g(x) w(x) \mathrm{d}x,
\]
on $ L^2_c(I,w(x)\mathrm{d}x)$, where $u_+ = 1$ and $u_- = \frac{1}{\sqrt{w p}}$.
Then $g\in\ker(\ell_\pm)$ implies
\be
f(x)=\begin{cases}
\int_a^x \sqrt{\frac{w(y)}{p(y)}} g(y) \mathrm{d}y \in \db(A_{0,+})\quad \text{when $g\in\ker(\ell_+)$}\\
\frac{1}{\sqrt{p(x) w(x)}} \int_x^b w(y) g(y) \mathrm{d}y \in \db(A_{0,-})\quad \text{when $g\in\ker(\ell_-)$}
\end{cases}
\ee
and $A_\pm f =g$.
Hence $\ell(g)=0$, i.e., $g\in\ker(\ell)$. So by \cite[Lemma~9.3]{tschroe} we conclude $\ell=\alpha_\mp \ell_\pm$, where $\alpha_\mp$ are constants.
And thus $h(x)=\tilde{h}(x)+\alpha_\mp u_\mp(x) \in\db(A_\mp)$ as required. 

If $\db(A_{0,\pm})$ were not dense, then for $h=0$ we could choose some nonzero $k\in\db(A_{0,\pm})^\perp$
and our calculation form above would give the contradiction $k= A_\mp \tilde{h} = A_\mp h =0$.
So we have shown $A_{0,\pm}^*=A_\mp$.

Finally, abbreviate $\db_\pm:= \{ f \in \db(A_\pm): \lim_{x\to,a,b} (\sqrt{w p} f)(x) g(x) = 0,\;\forall g\in\db(A_\mp)\}$. Then
\eqref{eqintbpart} shows $\db_\pm \subseteq \db(A_{0,\pm}^{**}) = \db(\ol{A_{0,\pm}})$. Conversely,
since $\ol{A_{0,\pm}}= A_\mp^* \subseteq A_\pm$, we can use \eqref{eqintbpart} to conclude
\[
\sqrt{w(b) p(b)}f(b)h(b) - \sqrt{w(a) p(a)}f(a)h(a)=0, \quad f\in\db(\ol{A_{0,\pm}}), h \in \db(A_\mp).
\]
Now replace $h$ by a $\tilde{h}\in\db(A_\mp)$ which coincides with $h$ near $a$ and
vanishes identically near $b$.
Then $$\sqrt{w(a) p(a)}f(a)h(a) = \sqrt{w(a) p(a)}f(a)\tilde{h}(a) - \sqrt{w(b) p(b)}f(b)\tilde{h}(b)=0.$$
Finally, $\sqrt{w(b) p(b)}f(b)h(b)= \sqrt{w(a) p(a)}f(a)h(a)=0$ shows $f\in\db_\pm$. 
\end{proof}

Based on this result, we define
\be
\qb := \db(A_+), \qquad
\qb_0 := \{ f \in \qb: \lim_{x\to a,b} (\sqrt{w p} g)(x) f(x) = 0,\;\forall g\in\db(A_-) \}.
\ee
Then we get two associated operators, the Friedrichs extension $A_- \ol{A_{0,+}}$ with form domain $\qb_0$ and
$\ol{A_{0,-}} A_+$ with form domain $\qb$. If the underlying differential expression
$-\frac{1}{w} \frac{\mathrm{d}}{\mathrm{d}x}p\frac{\mathrm{d}}{\mathrm{d}x}$ is limit point at both end points, then
both operators will of course agree.

In order to incorporate $q$, we will assume  $w,p\geq\delta$ for some constant $\delta>0$ such that $ L^2(I, w(x)\mathrm{d}x)\subseteq  L^2(I,\mathrm{d}x)$ and $\qb\subseteq H^1(I,\mathrm{d}x)$, where $H^1(I,\mathrm{d}x):=W^{1,2}(I,\mathrm{d}x)$ is the Sobolev space.

\begin{lemma}\label{qfree}
Suppose $w, p, q$ satisfy
\be\label{hypq}
w,p\geq \delta,\quad M:=\sup_{k\in\mathbb{Z}} \int_{[k,k+1)\cap I} |q(x)| \mathrm{d}x<\infty, 
\ee
where $\delta$ is a positive constant. If $I$ is compact, then the condition on $q$ amounts to requiring that $q$ is integrable.

Then $q/w$ is relatively form bounded with bound $0$ with respect to both $A_- \ol{A_{0,+}}$ and $\ol{A_{0,-}} A_+$, i.e., for every $\epsilon>0$,
$$\int_I|q(x)||f(x)|^2\mathrm{d}x\leq M\bigg(\epsilon\int_Ip(x)|f^\prime(x)|^2\mathrm{d}x+\bigg(1+\frac{1}{\epsilon}\bigg)\int_I|f(x)|^2w(x)\mathrm{d}x\bigg),\quad f\in\qb.$$
\end{lemma}

\begin{proof}
This follows literally as in Lemma~9.33 from \cite{tschroe}.
\end{proof}

Hence we can define either $T = A_- \ol{A_{0,+}} + q/w$
with $\db(T) = \{ f \in \qb_0 \mid \ol{A_{0,+}} f \in \db(A_-)\}$ and $\qb(T) =\qb_0$ or $T = \ol{A_{0,-}} A_+ + q/w$ with $\db(T) = \{ f \in \qb \mid A_+ f \in \db(\ol{A_{0,-}})\}$ and $\qb(T) =\qb$ as form sums such that the associated
quadratic form is
\be
\int_a^b \big( p(x) f'(x)^* g'(x) + q(x) f(x)^* g(x)\big) \mathrm{d}x, \qquad f,g\in\qb(T).
\ee

\begin{theorem}\label{slnrc}
Suppose $w,p,q$ and $w_n,p_n,q_n$ both satisfy {\rm Hypothesis~\ref{hyp}}, \eqref{hypq} and $q_n,q\geq \gamma$ for some constant $\gamma\in\R$. If
\be
w_n/w \to 1, \quad p_n/p \to 1, \qquad C_n:=\sup_{k\in\mathbb{Z}} \int_{[k,k+1)\cap I} |q_n(x) - q(x)| \mathrm{d}x \to 0,
\ee
uniformly, then we have $T_n\rightarrow T$ in generalized norm resolvent convergence.
\end{theorem}

\begin{proof}
It is easy to verify 
\begin{equation*}
    \begin{split}
        J_n:\quad   L^2(I,w(x)\mathrm{d}x)&\rightarrow L^2(I,w_n(x)\mathrm{d}x)\\
        f&\mapsto f,
    \end{split}
\end{equation*}
satisfies (\ref{Jcon}).
   Since  $$\int_I|f(x)|^2w_n(x)\mathrm{d}x=\int_I|f(x)|^2w(x)\frac{w_n(x)}{w(x)}\mathrm{d}x,$$ we have $J_n\qb(T)\subseteq \qb(T_n)$. Let $T=T^{(1)}+T^{(2)}$, where $T^{(1)}:=\frac{1}{w}(-\frac{\mathrm{d}}{\mathrm{d}x}p\frac{\mathrm{d}}{\mathrm{d}x})$ and $T^{(2)}:=\frac{q}{w}$ in $ L^2(I,w(x)\mathrm{d}x)$. For any $\psi\in\qb(T)$,
   \begin{equation*}
       \begin{split}
           |\textbf{q}_{T^{(1)}}(\psi)-\textbf{q}_{T_n^{(1)}}(J_n\psi)|&=\left|\int_I\bigg(\frac{p_n(x)}{p(x)}-1\bigg)p(x)|\psi^\prime(x)|^2\mathrm{d}x\right|\\
           &\leq \bigg\|\frac{p_n}{p}-1\bigg\|_{\infty}\int_Ip(x)|\psi^\prime(x)|^2\mathrm{d}x,
       \end{split}
   \end{equation*}
   and by Lemma \ref{qfree},
   \begin{equation*}
       \begin{split}
           &|\textbf{q}_{T^{(2)}}(\psi)-\textbf{q}_{T_n^{(2)}}(J_n\psi)|\leq\int_I\left|q(x)-q_n(x)\right||\psi(x)|^2\mathrm{d}x\\
           &\leq C_nC(p,w)\bigg(\epsilon\int_Ip(x)|\psi^\prime(x)|^2\mathrm{d}x+\bigg(1+\frac{1}{\epsilon}\bigg)\int_I|\psi(x)|^2w(x)\mathrm{d}x\bigg),
       \end{split}
   \end{equation*}
   where $C(p,w)$ is a positive constant associated with $p$ and $w$. 
   
   Note that $T_n,T\geq \gamma$ are bounded from below. Therefore, using Lemma \ref{q-nrc} we finally have $T_n\rightarrow T$ in generalized norm resolvent sense.
\end{proof}

The following result generalizes Theorem~3.5 from \cite{BSTT23}.

\begin{theorem}
Suppose $I=\R$, $w_0,p_0,q_0$ and $w_\infty,p_\infty,q_\infty$ both satisfy {\rm Hypothesis~\ref{hyp}}, \eqref{hypq} and $q_0, q_\infty\geq \gamma$ for some constant $\gamma\in\R$. If
$$
w_\infty(x)/w_0(x) \to 1, \quad p_\infty(x)/p_0(x) \to 1, \quad \text{as $|x|\to\infty$,}
$$
and
$$
\int_k^{k+1} |q_\infty(x) - q_0(x)| \mathrm{d}x \to 0,\quad \text{as $|k|\rightarrow \infty$.}
$$ 
Then $\sig_{ess}(T_\infty)=\sig_{ess}(T_0)$.
\end{theorem}

\begin{proof}
Define
\[
w_n(x) = \begin{cases} w_\infty(x), & |x| \le n,\\  w_0(x), & |x| > n \end{cases}
\]
and similarly for $p_n$ and $q_n$. Then it is well-known that $\sig_{ess}(T_n)=\sig_{ess}(T_0)$ since the essential spectrum
does not change when the coefficients are changed on a compact set:
Adding extra (e.g.) Dirichlet boundary conditions at $-n$ and $n$ splits the operator into a direct sum of three operators which differ from the original one by a rank two perturbation. Hence by Weyl's theorem the essential spectrum of the original operator is the union of the essential spectra of the three parts. As the middle part corresponds to a regular opeartor it does not contribute to the essentials spectrum.

Using Theorem \ref{slnrc} we have $T_n\to T_\infty$ in generalized norm resolvent sense and the claim follows.
\end{proof}

\bigskip
\noindent
\textbf{Acknowledgments.}
Y. Wang and Z. Zhou thank the Faculty of Mathematics of the University of Vienna for the hospitality during extended visits in 2025.

Y. Wang and Z. Zhou were supported in part by the National Natural Science Foundation of China( No. 12271509, 12090010, 12090014).

\end{document}